\documentclass{ws-ijbc}
\usepackage{ws-rotating}     
\usepackage{graphicx}
\usepackage{epstopdf}

\usepackage{amsfonts, amsbsy, amssymb, amsmath, graphicx, float, natbib, color, url, titlesec}
\newtheorem{Lem}[theorem]{Lemma}
\newtheorem{assumption}{Assumption}
\usepackage[english]{babel}

\begin{document}

\markboth{Balibrea-Iniesta et al.}{}

\title{\Large{\textbf{CHAOTIC DYNAMICS IN NONAUTONOMOUS MAPS: \\ APPLICATION TO THE NONAUTONOMOUS H\'ENON MAP}}}

\author{FRANCISCO BALIBREA-INIESTA\footnote{Instituto de Ciencias Matem\'aticas, CSIC-UAM-UC3M-UCM, Spain, \textit{francisco.balibrea@icmat.es}.}}

\address{Instituto de Ciencias Matem\'aticas, CSIC-UAM-UC3M-UCM,\\ C/ Nicol\'as Cabrera 15, Campus de Cantoblanco UAM\\
Madrid, 28049, Spain\\
francisco.balibrea@icmat.es}

\author{CARLOS LOPESINO\footnote{Instituto de Ciencias Matem\'aticas, CSIC-UAM-UC3M-UCM, Spain, \textit{carlos.lopesino@icmat.es}.}}

\address{Instituto de Ciencias Matem\'aticas, CSIC-UAM-UC3M-UCM,\\ C/ Nicol\'as Cabrera 15, Campus de Cantoblanco UAM\\
Madrid, 28049, Spain\\
carlos.lopesino@icmat.es}

\author{STEPHEN WIGGINS\footnote{School of Mathematics, University of Bristol, United Kingdom, \textit{s.wiggins@bristol.ac.uk}.}}

\address{School of Mathematics, University of Bristol,\\ Bristol BS8 1TW, United Kingdom \\
s.wiggins@bristol.ac.uk}

\author{ANA M. MANCHO\footnote{Instituto de Ciencias Matem\'aticas, CSIC-UAM-UC3M-UCM, Spain, \textit{a.m.mancho@icmat.es}.}}

\address{Instituto de Ciencias Matem\'aticas, CSIC-UAM-UC3M-UCM,\\ C/ Nicol\'as Cabrera 15, Campus de Cantoblanco UAM\\
Madrid, 28049, Spain\\
a.m.mancho@icmat.es}

\maketitle

\begin{abstract}
\noindent
In this paper we analyze chaotic dynamics for two dimensional nonautonomous maps through the use of a nonautonomous version of the Conley-Moser conditions given previously. With this approach we are able to give a precise definition of what is meant by a {\em chaotic invariant set} for nonautonomous maps. We extend the nonautonomous Conley-Moser conditions by deriving a new sufficient condition for the nonautonomous chaotic invariant set to be hyperbolic.  We consider the specific example of a nonautonomous H\'enon map and give sufficient conditions, in terms of the parameters defining the map, for the nonautonomous H\'enon map to have a hyperbolic chaotic invariant set.
\newline
\\
\textit{Keywords:} chaotic dynamics, invariant set, hyperbolic.
\end{abstract}

\section{Introduction}
\label{sec:intro}

Studies of the H\'{e}non map  (\cite{Henon}) have played a seminal role in the development of our understanding of chaotic dynamics and strange attractors.
The map depends on two parameters, $A$ and $B$, and has the following form:

\begin{equation}
\begin{array}{ccccc}
H & : & \quad \mathbb{R}^{2} & \quad \longrightarrow & \mathbb{R}^{2}, \\ & & (x,y) & \quad \longmapsto & (A+By-x^{2},x),
\end{array}
\end{equation}

\noindent
where we will require $B \not = 0$ in order to endure  the existence of the inverse map,

\begin{equation}
\begin{array}{ccccc}
H^{-1} & : & \quad \mathbb{R}^{2} & \quad \longrightarrow & \mathbb{R}^{2}, \\ & & (x,y) & \quad \longmapsto & (y,(x-A+y^{2})/B).
\end{array}
\end{equation}

The ``heart'' of chaotic dynamics is exemplified by the so-called ``Smale horseshoe map'' (see \cite{s80} for a general description, with background). The essential feature of the Smale horseshoe map for chaos is that the map contains an invariant Cantor set on which the dynamics are topologically conjugate to a shift map defined on a finite number of symbols (a ``chaotic invariant set'', sometimes also referred to as a ``chaotic saddle''). \cite{Dev79}  gave sufficient conditions, in terms of the parameters $A$ and $B$, for the H\'enon map to have an invariant Cantor set on which it is topologically conjugate to a shift map of two symbols. The proof uses a technique due to Conley and Moser (see \cite{Moser})  that is referred to as the ``Conley-Moser conditions'' (but for earlier work in a similar spirit see \cite{vma_a,vma_b,vma_c}). \cite{Holmes} used these conditions to show the existence of a chaotic invariant set in the so-called ``bouncing ball map''. The Conley-Moser conditions were given a more detailed exposition, along with a slight weakening of the hypotheses, in \cite{Wiggins03}. More recently, the  Conley-Moser conditions were used to show the existence of a chaotic invariant set in the Lozi map (\cite{Carlos}).

The purpose of this paper is to carry out a similar analysis for a nonautonomous version of the H\'enon map. The generalization of the Conley-Moser conditions for  nonautonomous systems, i.e. in the discrete time setting with dynamics  defined by an infinite sequence of maps, was given in \cite{Wiggins99}. We extend the nonautonomous Conley-Moser conditions further by providing an additional condition which is sufficient for the nonautonomous chaotic invariant set to be hyperbolic. Hyperbolicity of nonautonomous invariant sets is discussed in general in \cite{KH}. Earlier work on chaos in nonautonomous systems  is described in \cite{ls,stoffa,stoffb}. Recent interesting work is described in \cite{lw1,lw2}.

While the development of the ``dynamical systems approach to nonautonomous dynamics'' is currently a topic of much interest, it is not a topic that is widely known in the applied dynamical systems community (especially the fundamental work that was done in the 1960's). An applied motivation for such work is an understanding of fluid transport from the dynamical systems point of view for aperiodically time dependent flows. \cite{wm} have given a survey of the history of nonautonomous dynamics as well as its application to fluid transport. 

This paper is outlined as follows. In Section \ref{sec:pc} we develop the required concepts for ``building'' chaotic invariant sets for two-dimensional nonautonomous maps. In Section \ref{sec:mainthm} we  prove the ``main theorem'' generalizing the Conley-Moser conditions that provide necessary conditions for two-dimensional nonautonomous maps to have a chaotic invariant set.  In the course of the proof of the theorem the nature of chaotic invariant sets, and chaos, for nonautonomous maps is developed. This theorem was first given in \cite{Wiggins99}, but in Section \ref{sec:nacm3} we develop the theory further by providing a more  analytical, rather than topological, construction for one of the Conley-Moser conditions that allows us to conclude that the nonautonomous chaotic invariant set is hyperbolic. In Section \ref{sec:nahmap} we develop a version of the nonautonomous H\'enon map and use the previously developed results to give sufficient conditions for the map to possess a nonautonomous chaotic invariant set. In Section \ref{sec:summ} we discuss directions for future work along these lines.

\section{Preliminary concepts}
\label{sec:pc}

In this section we describe the basic setting and concepts that we will use throughout the remainder of the paper. 

Nonautonomous dynamics will be defined by a sequence of maps and domains,
$\lbrace f_{n},D_{n} \rbrace_{n=-\infty}^{+ \infty}$, acting as follows:

\begin{equation} f_{n}: D_{n} \longrightarrow D_{n+1} \quad \quad \forall n \in \mathbb{Z} \quad \text{and} \quad f^{-1}_{n}: D_{n+1} \longrightarrow D_{n}, \end{equation}

\noindent
where, for our purposes, $D_n$ will be an appropriately chosen domain in $\mathbb{R}^2$, for all $n$.

Similar to the Smale horseshoe construction (\cite{Wiggins03}), on each domain $D_{n}$ we must construct  a finite collection of vertical strips $V_{i}^{n} \subset D_{n}$ ($\forall n \in \mathbb{Z}$ and $\forall i \in I = \lbrace 1,2,...,N \rbrace $) which map to a finite collection of horizontal strips $H_{i}^{n+1}$ located in $D_{n+1}$:

\begin{equation} H_{i}^{n+1} \subset D_{n+1} \quad \text{with} \quad f_{n}(V_{i}^{n})=H_{i}^{n+1} \text{ ,} \quad \forall n \in \mathbb{Z} \text{ ,} \quad i\in I. \end{equation}

\noindent
Associated with these mappings we will need to  define a {\em transition matrix} as follows:

$$ A \equiv \lbrace A^{n} \rbrace_{n=-\infty}^{+\infty} \text{ is a sequence of matrices of dimension } N \times N \text{ such that}$$
$$ A_{ij}^{n}= \begin{cases} 1 \quad \quad \text{if } f_{n}(V_{i}^{n}) \cap V_{j}^{n+1} \not = \emptyset \\ 0 \quad \quad \text{otherwise} \end{cases} \text{or equivalently} $$
\begin{equation} A_{ij}^{n}= \begin{cases} 1 \quad \quad \text{if } H_{i}^{n+1} \cap V_{j}^{n+1} \not = \emptyset \\ 0 \quad \quad \text{otherwise} \end{cases} \forall i,j \in I. \hspace{1.65cm} \end{equation}

However, first we must precisely define the notion of the domains that we will use, horizontal and vertical strips in those domains, and provide a characterization of the intersection of horizontal and vertical strips in the domain appropriate for our purposes. 

To begin, let $D \subset \mathbb{R}^{2}$ denote a closed and bounded set. We consider two associated subsets of $\mathbb{R}$:

$$ D_{x} = \lbrace x \in \mathbb{R} \text{ } | \text{ } \text{there exists a } y \in \mathbb{R} \text{ with } (x,y) \in D \rbrace $$
\begin{equation} D_{y} = \lbrace y \in \mathbb{R} \text{ } | \text{ } \text{there exists an } x \in \mathbb{R} \text{ with } (x,y) \in D \rbrace \end{equation}

\noindent
Therefore $D_{x}$ and $D_{y}$ represent the projections of $D$ onto the $x$-axis and the $y$-axis respectively. From this it is easy to see that $D \subset D_{x} \times D_{y}$. We consider two closed intervals $I_{x} \subset D_{x}$ and $I_{y} \subset D_{y}$. We next define $\mu_h$-horizontal and $\mu_v$-vertical curves on these domains.

\begin{definition} Let $0 \leq \mu_{h} < + \infty$. A $\mu_{h}$-horizontal curve $\overline{H}$ is defined to be the graph of a function $h : I_{x} \rightarrow \mathbb{R}$ where $h$ satisfies the following two conditions:
\newline
\\
1. The set $\overline{H}= \lbrace (x,h(x)) \in \mathbb{R}^{2} \text{ } | \text{ } x \in I_{x} \rbrace$ is contained in $D$.
\\
2. For every $x_{1},x_{2} \in I_{x}$ we have the Lipschitz condition
\begin{equation} |h(x_{1})-h(x_{2})| \leq \mu_{h} |x_{1}-x_{2}| \end{equation}
\\
Similarly, let $0 \leq \mu_{v} < + \infty$. A $\mu_{v}$-vertical curve $\overline{V}$ is defined to be the graph of a function $v : I_{y} \rightarrow \mathbb{R}$ where $v$ satisfies the following two conditions:
\newline
\\
1. The set $\overline{V}= \lbrace (v(y),y) \in \mathbb{R}^{2} \text{ } | \text{ } y \in I_{y} \rbrace$ is contained in $D$.
\\
2. For every $y_{1},y_{2} \in I_{y}$ we have the Lipschitz condition
\begin{equation} |v(y_{1})-v(y_{2})| \leq \mu_{v} |y_{1}-y_{2}| \end{equation} \end{definition}

\noindent
Next we ``fatten'' these curves into strips.

\begin{definition} Given two nonintersecting $\mu_{v}$-vertical curves $v_{1}(y)<v_{2}(y)$, $y \in I_{y}$, we define a $\mu_{v}$-vertical strip as
\begin{equation} V = \lbrace (x,y) \in \mathbb{R}^{2} \text{ } | \text{ } x \in [v_{1}(y),v_{2}(y)], \text{ } y \in I_{y} \rbrace \end{equation}
\\
Similarly, given two nonintersecting $\mu_{h}$-horizontal curves $h_{1}(x)<h_{2}(x)$, $x \in I_{x}$, we define a $\mu_{h}$-horizontal strip as
\begin{equation} H = \lbrace (x,y) \in \mathbb{R}^{2} \text{ } | \text{ } y \in [h_{1}(x),h_{2}(x)], \text{ } x \in I_{x} \rbrace \end{equation}
\newline
The width of horizontal and vertical strips is defined as
\begin{equation} d(H)= \max_{x \in I_{x}} |h_{2}(x)-h_{1}(x)| \quad , \quad d(V)=\max_{y \in I_{y}} |v_{2}(y)-v_{1}(y)| \end{equation} \end{definition}

\noindent
We will need to consider different parts of the  boundary of the strips in relation to the domain on which they are defined. The following three definitions provide the necessary concepts.

\begin{definition} The vertical boundary of a $\mu_{h}$-horizontal strip $H$ is denoted
\begin{equation} \partial_{v}H \equiv \lbrace (x,y) \in H \text{ } | \text{ } x \in \partial I_{x} \rbrace \end{equation}
The horizontal boundary of a $\mu_{h}$-horizontal strip $H$ is denoted
\begin{equation} \partial_{h}H \equiv \partial H \setminus \partial_{v}H \end{equation} \end{definition}

\begin{definition} We say that $H$ is a $\mu_{h}$-horizontal strip contained in a $\mu_{v}$-vertical strip $V$ if the two $\mu_{h}$ horizontal curves defining the horizontal boundaries of $H$ (denoted by $\partial_{h}H$) are contained in $V$, with the remaining boundary components of $H$ (denoted by $\partial_{v}H$) contained in $\partial_{v}V$. These two last subsets, $\partial_{h}H$ and $\partial_{v}H$ are referred to as the horizontal and vertical boundaries of $H$, respectively. See Figure 1.
\end{definition}

\begin{figure}[h!]
\label{fig1}
\centering
\includegraphics[scale=0.38]{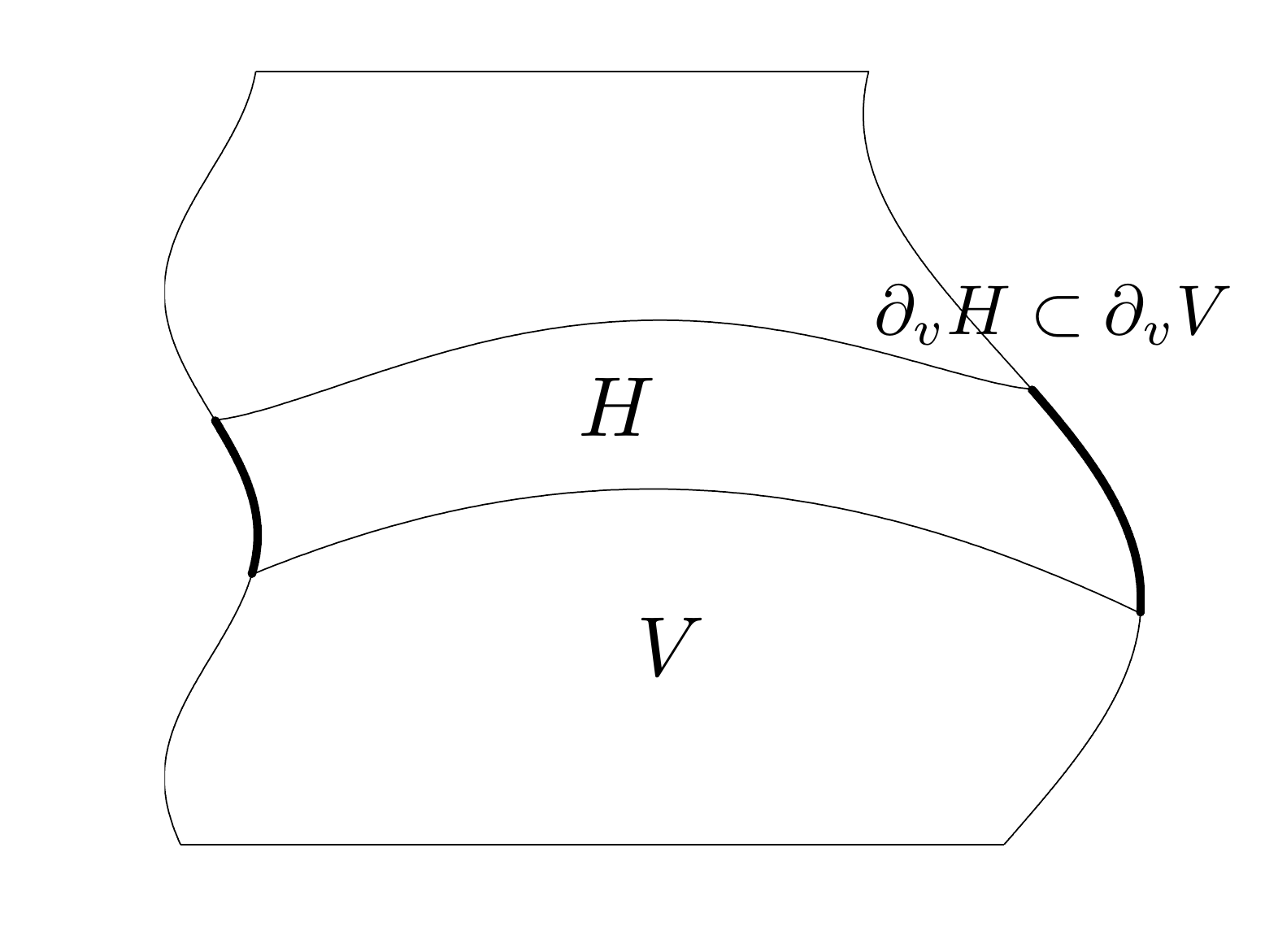}
\caption{$H$ is bounded by two $\mu_{h}$-horizontal curves, each of them linking the two $\mu_{v}$-vertical curves composing $\partial_{v}V$.}
\end{figure}

\begin{definition} Let $V$ and $\widetilde{V}$ be $\mu_{v}$-vertical strips. $\widetilde{V}$ is said to intersect $V$ fully if $\widetilde{V} \subset V$ and $\partial_{h}\widetilde{V} \subset \partial_{h}V$. See Figure \ref{fig2}.
\end{definition}

\begin{figure}[h!]
\label{fig2}
\centering
\includegraphics[scale=0.5]{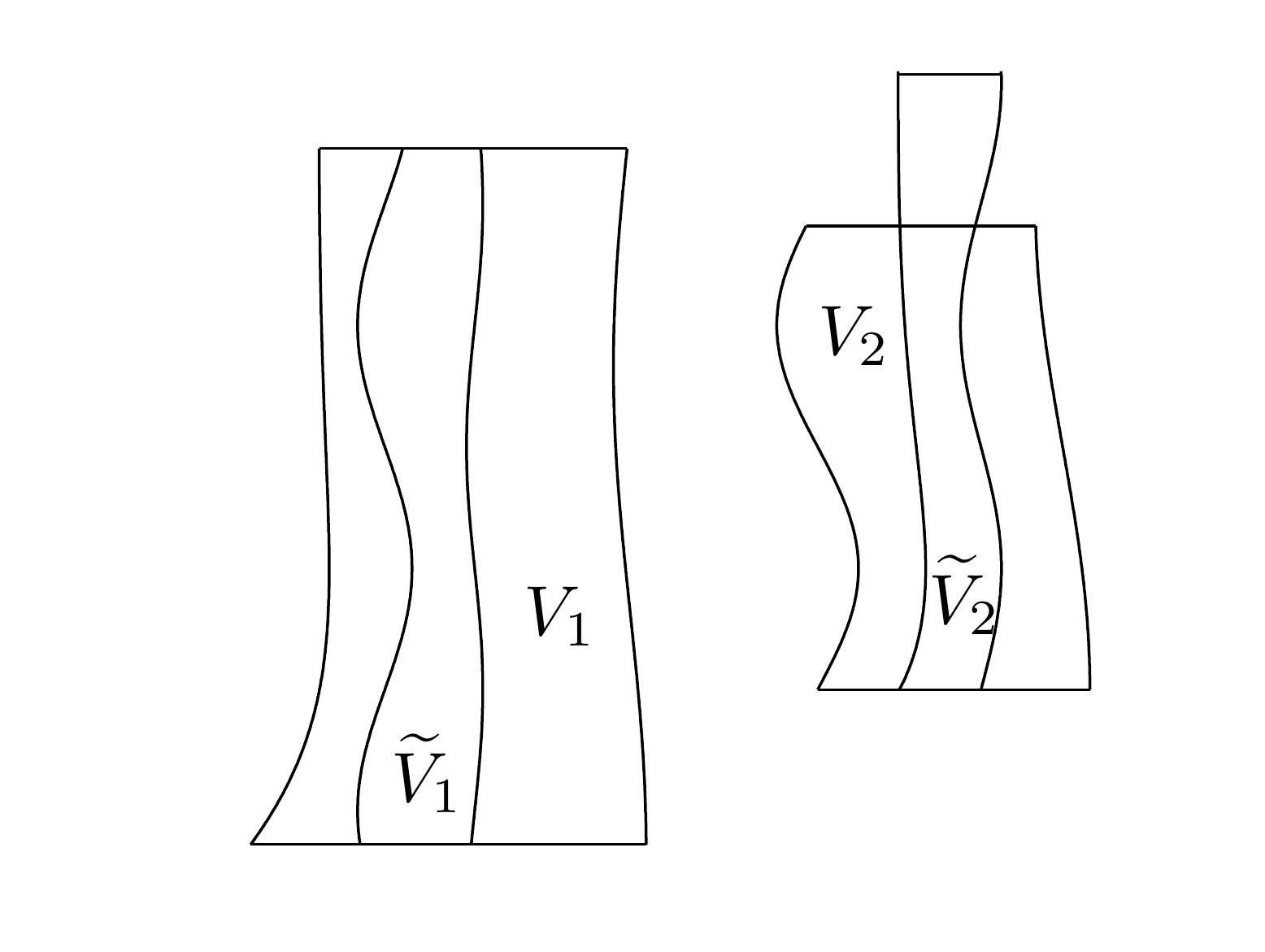}
\caption{$\widetilde{V}_{1}$ intersects $V_{1}$ fully. This does not happen for $\widetilde{V}_{2}$ and $V_{2}$.}
\end{figure}

\section{The main theorem}
\label{sec:mainthm}

In this section we prove the main general theorem which provides sufficient conditions for the existence of a chaotic invariant set for  nonautonomous maps. In the course of the proof our meaning of ``chaos'' for nonautonomous dynamics will be made precise. 

Following  the original development of the Conley-Moser conditions (\cite{Moser}), there are three  geometrical and analytical conditions that, if satisfied, provide  sufficient conditions for an autonomous map (in the original formulation) to have a chaotic invariant set. These are referred to as A1, A2, and A3. The conditions A1 and A2 provide sufficient conditions for the existence of a topological chaotic invariant set. The conditions A1 and A3 provide sufficient conditions for a hyperbolic chaotic invariant set.
Conditions A1 and A2 were developed for nonautonomous dynamics in \cite{Wiggins99}.  In this section we recall A1 and A2, but we also give a new construction of A3 for nonautonomous dynamics\footnote{We point out a minor technical point. In previous development of the Conley-Moser conditions (e.g. \cite{Moser, Wiggins03}) the set-up considers the mapping of horizontal strips to vertical strips. However, for the H\'enon map it is more natural to consider vertical strips mapping to horizontal strips. Of course, the choice of what we refer to as ``horizontal'' and ``vertical'' is arbitrary.  However, the same choice of coordinate labeling  as is used in the previous literature can be used for  the H\'enon map if we  impose a rotation $P = \left( \begin{array}{cc} 0 & 1 \\ 1 & 0 \end{array} \right) $ on the sequence of maps $\lbrace f_{n} \rbrace_{n=-\infty}^{+\infty}$ or, alternatively, take each map $f_{n}$ as $f_{-n}^{-1}$ for every $n \in \mathbb{Z}$.}. In particular, we will show that A1 and A3 imply that A1 and  A2 also hold.

The following two lemmas will play an important role in the proof of the main theorem.

\begin{Lem} i) If $V_{1} \supset V_{2} \supset \cdots \supset V_{k} \supset \cdots$ is a nested sequence of $\mu_{v}$-vertical strips with $d(V_{k}) \rightarrow 0$ as $k \rightarrow \infty$, then $\cap_{k=1}^{\infty} V_{k} \equiv V_{\infty}$ is a $\mu_{v}$-vertical curve.
\newline
\\
ii) If $H_{1} \supset H_{2} \supset \cdots \supset H_{k} \supset \cdots$ is a nested sequence of $\mu_{h}$-horizontal strips with $d(H_{k}) \rightarrow 0$ as $k \rightarrow \infty$, then $\cap_{k=1}^{\infty} H_{k} \equiv H_{\infty}$ is a $\mu_{h}$-horizontal curve. \label{lem1}
\end{Lem}

\begin{Lem} Suppose $0 \leq \mu_{v} \mu_{h} <1$. Then a $\mu_{v}$-vertical curve and a $\mu_{h}$-horizontal curve intersect in a unique point. 
\label{lem2}
\end{Lem}

\noindent
The proof of both these two lemmas can be found in \cite{Wiggins03}.

We assume that for each  $D_{n} \subset \mathbb{R}^{2}$ we have:

\begin{equation} 
f_{n}(D_{n}) \cap D_{n+1} \not = \emptyset \text{ ,} \quad \forall n \in \mathbb{Z} \end{equation}

Furthermore, we assume that on each $D_n$ we can find a set of disjoint $\mu_v$ vertical strips, $D_{V}^{n} \equiv \cup_{i=1}^{N}V_{i}^{n}$, such that 
 each $f_{n}$ is one-to-one on $D_{V}^{n} \equiv \cup_{i=1}^{N}V_{i}^{n}$. We then define

\begin{equation} H_{ij}^{n+1} \equiv f_{n}(V_{i}^{n}) \cap V_{j}^{n+1} = H_{i}^{n+1} \cap V_{j}^{n+1}, \quad \text{and} \quad V_{ji}^{n} \equiv f_{n}^{-1}(V_{j}^{n+1}) \cap V_{i}^{n} \end{equation}

\noindent
with inverse function $f_{n}^{-1}$ defined on $D_{H}^{n+1} \equiv \cup_{i=1}^{N}H_{i}^{n+1} = f_{n} \left( \cup_{i=1}^{N}V_{i}^{n} \right)$ for every $n \in \mathbb{Z}$, (see Figure \ref{fig3}).

\noindent
The transition matrix $\lbrace A^{n} \rbrace_{n=-\infty}^{+\infty}$ is defined as follows:

\begin{equation} A_{ij}^{n}= \begin{cases} 1 \quad \quad \text{if } H_{ij}^{n+1} = H_{i}^{n+1} \cap V_{j}^{n+1} \not = \emptyset \\ 0 \quad \quad \text{otherwise} \end{cases} \forall i,j \in I. \hspace{1.65cm} \end{equation}

Now we can state the first two Conley-Moser conditions for a sequence of maps that are sufficient to prove the existence of a chaotic invariant set for nonautonomous systems.

\begin{assumption}[\textbf{A1}] For all $i,j \in I$ such that $A_{ij}^{n}=1$, $H_{ij}^{n+1}$ is a $\mu_{h}$-horizontal strip contained in $V_{j}^{n+1}$ with $0 \leq \mu_{v} \mu_{h} <1$. Moreover, $f_{n}$ maps $V_{ji}^{n}$ homeomorphically onto $H_{ij}^{n+1}$ with $f_{n}^{-1}(\partial_{h}H_{ij}^{n+1}) \subset \partial_{h}V_{i}^{n}$. 
\end{assumption}

\begin{figure}[h!]
\centering
\includegraphics[scale=0.55]{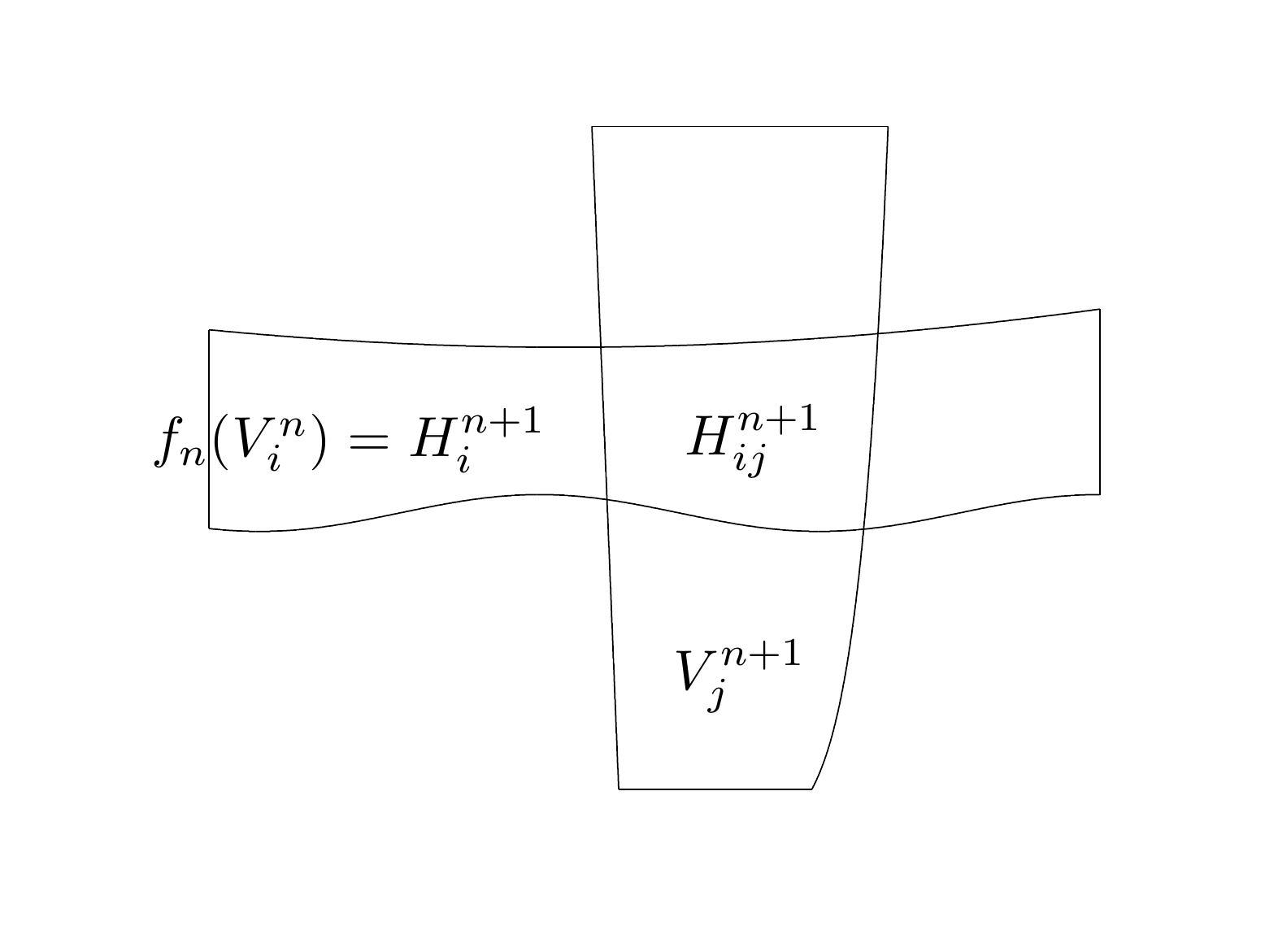}
\caption{Assuming that A1 is satisfied for a given sequence of maps, this figure illustrates that every non empty $H_{ij}^{n+1} \subset D_{n+1}$ is a $\mu_{h}$-horizontal strip contained in $V_{j}^{n+1}$. This also shows that the two $\mu_{h}$-horizontal curves which form the boundary $\partial_{h}f_{n}(V_{i}^{n})=\partial_{h}(H_{i}^{n+1})$ cut the vertical boundary of $V_{j}^{n+1}$ in exactly four points.}
\label{fig3}
\end{figure}

\noindent
Also, since $f_{n}$ maps $V_{ji}^{n}$ homeomorphically onto $H_{ij}^{n+1}$ with $f_{n}^{-1}(\partial_{h}H_{ij}^{n+1}) \subset \partial_{h}V_{i}^{n}$ then $f_{n}^{-1}$ maps $H_{ij}^{n+1}$ homeomorphically onto $V_{ji}^{n}$ ($\forall i,j \in I$) with

\begin{equation} 
f_{n} \left( f_{n}^{-1}(\partial_{h}H_{ij}^{n+1}) \right) =  \partial_{h}H_{ij}^{n+1} \subset f_{n}(\partial_{h} V_{i}^{n}). 
\end{equation}

\begin{assumption}[\textbf{A2}] Let $V^{n+1}$ be a $\mu_{v}$-vertical strip which intersects $V_{j}^{n+1}$ fully. Then $f_{n}^{-1}(V^{n+1}) \cap V_{i}^{n} \equiv \widetilde{V}_{i}^{n}$ is a $\mu_{v}$-vertical strip intersecting $V_{i}^{n}$ fully for all $i \in I$ such that $A_{ij}^{n}=1$. Moreover,
\begin{equation} d(\widetilde{V}_{i}^{n}) \leq \nu_{v} \text{ } d(V^{n+1}) \quad \quad \text{for some } 0< \nu_{v} < 1 \end{equation}
\\
Similarly, let $H^{n}$ be a $\mu_{h}$-horizontal strip contained in $V_{j}^{n}$ such that also $H^{n} \subset H_{ij}^{n}$ for some $i,j \in I$ with $A_{ij}^{n-1}=1$. Then $f_{n}(H^{n}) \cap V_{k}^{n+1} \equiv \widetilde{H}_{k}^{n+1}$ is a $\mu_{h}$-horizontal strip contained in $V_{k}^{n+1}$ for all $k \in I$ such that $A_{jk}^{n}=1$. Moreover,
\begin{equation} d(\widetilde{H}_{k}^{n+1}) \leq \nu_{h} \text{ } d(H^{n}) \quad \quad \text{for some } 0< \nu_{h} < 1 \end{equation}
\end{assumption}

Now we develop  symbolic dynamics in a form appropriate for nonautonomous dynamics. Let

\begin{equation} 
s=( \cdots s_{n-k} \cdots s_{n-2} s_{n-1}. s_{n} s_{n+1} \cdots s_{n+k} \cdots ) \end{equation}

\noindent
denote a bi-infinite sequence with $s_{l} \in I$ ($\forall l\in \mathbb{Z}$) where adjacent elements of the sequence satisfy the rule $A_{s_{n}s_{n+1}}^{n}=1$, $\forall n \in \mathbb{Z}$.

\noindent
Similarly to the symbolic dynamics implemented for the Smale horseshoe (see page 575 of \cite{Wiggins03}), here we denote the set of all such symbol sequences by $\Sigma_{ \lbrace A^{n} \rbrace }^{N}$. If $\sigma$ denotes the shift map

\begin{equation} \sigma (s)= \sigma ( \cdots s_{n-2} s_{n-1}. s_{n} s_{n+1} \cdots ) =  ( \cdots s_{n-2} s_{n-1} s_{n}. s_{n+1} \cdots )\end{equation}

\noindent
on $\Sigma_{ \lbrace A^{n} \rbrace }^{N}$, we define the ``extended shift map'' $\tilde{\sigma}$ on $\widetilde{\Sigma} \equiv \Sigma_{ \lbrace A^{n} \rbrace }^{N} \times \mathbb{Z}$ by

\begin{equation} \
\tilde{\sigma}(s,n)=(\sigma (s),n+1). \quad \text{It is also defined } f(x,y;n)=(f_{n}(x,y),n+1). \end{equation}

Now we can state the main theorem.

\begin{theorem}[Main theorem] Suppose $\lbrace f_{n},D_{n} \rbrace_{n=-\infty}^{+\infty}$ satisfies A1 and A2. There exists a sequence of sets $\Lambda_{n} \subset D_{n}$, with $f_{n}(\Lambda_{n})=\Lambda_{n+1}$, such that the following diagram commutes
\begin{equation} \begin{array}{ccccc} & & f & & \\ & \Lambda_{n} \times \mathbb{Z} & \longrightarrow & \Lambda_{n+1} \times \mathbb{Z} & \\ \\ \phi & \downarrow & & \downarrow & \phi \\ & & \tilde{\sigma} & & \\ & \Sigma_{ \lbrace A^{n} \rbrace }^{N} \times \mathbb{Z} & \longrightarrow & \Sigma_{ \lbrace A^{n} \rbrace }^{N} \times \mathbb{Z} & \end{array} 
\end{equation}

\noindent
where $\phi (x,y;n) \equiv (\phi_{n}(x,y),n)$ with $\phi_{n}(x,y)$ a homeomorphism mapping $\Lambda_{n}$ onto $\Sigma_{ \lbrace A^{n} \rbrace }^{N}$.
\label{mainthm}
\end{theorem}

\noindent
\begin{remark} The sequence of sets $\lbrace \Lambda_{n} \rbrace_{n=-\infty}^{+\infty}$ is what we mean by a \textbf{chaotic set} for nonautonomous dynamics. Consequently our ``main theorem'' is a theoretical result which gives sufficient conditions for the existence of such a sequence of sets. The original proof can be found in \cite{Wiggins99}, keeping in mind the geometrical considerations mentioned before.
\end{remark}
Next, we will generalize the third Conley-Moser condition to the nonautonomous case. This will provide an alternative, and  more analytical (as opposed to topological) method for proving  that the  second Conley-Moser condition holds, and it will also provide the additional information that the chaotic invariant set is hyperbolic.

\subsection{Nonautonomous third Conley-Moser condition}
\label{sec:nacm3}

We begin by giving a natural definition of stable and unstable sector bundles for the nonautonomous situation:

\begin{equation} { \cal V}^{n} \equiv \bigcup_{i,j \in I}V_{ji}^{n} = \bigcup_{i,j \in I} f_{n}^{-1}(V_{j}^{n+1}) \cap V_{i}^{n}, \end{equation}

\begin{equation} { \cal H}^{n+1} \equiv \bigcup_{i,j \in I}H_{ij}^{n+1} = \bigcup_{i,j \in I} H_{i}^{n+1} \cap V_{j}^{n} , \quad f_{n}({ \cal V}^{n})={ \cal H}^{n+1} \end{equation}

\begin{equation} S^{u}_{{ \cal K}} \equiv \lbrace (\xi_{z},\eta_{z}) \in \mathbb{R}^{2} \text{ } | \text{ } | \eta_{z} | \leq \mu_{h} | \xi_{z} | \text{, } z \in { \cal K} \rbrace \quad \text{(unstable sector bundle)} \end{equation}

\begin{equation} S^{s}_{{ \cal K}} \equiv \lbrace (\xi_{z},\eta_{z}) \in \mathbb{R}^{2} \text{ } | \text{ } | \xi_{z} | \leq \mu_{v} | \eta_{z} | \text{, } z \in { \cal K} \rbrace \quad \text{(stable sector bundle)} \end{equation}

\noindent
with ${ \cal K}$ being either ${ \cal V}^{n}$ or ${ \cal H}^{n+1}$. Then we can state the following assumption: the third Conley-Moser condition for the nonautonomous setting.

\begin{assumption}[\textbf{A3}] $Df_{n}(S^{u}_{{ \cal V}^{n}})\subset S^{u}_{ { \cal H}^{n+1}}$, $Df^{-1}_{n}(S^{s}_{{ \cal H}^{n+1}})\subset S^{s}_{{ \cal V}^{n}}$.
\newline
\\
Moreover, if $(\xi_{f_{n}(z_{0}^{n})}, \eta_{f_{n}(z_{0}^{n})}) \equiv Df_{n}(z_{0}^{n}) \cdot (\xi_{z_{0}^{n}}, \eta_{z_{0}^{n}}) \in S^{u}_{{ \cal H}^{n+1}}$ then 

\begin{equation} | \xi_{f_{n}(z_{0}^{n})} | \geq \left( \frac{1}{\mu} \right) | \xi_{z_{0}^{n}} | \end{equation}

\noindent
If $(\xi_{f_{n}^{-1}(z_{0}^{n+1})}, \eta_{f_{n}^{-1}(z_{0}^{n+1})}) \equiv Df_{n}^{-1}(z_{0}^{n+1}) \cdot (\xi_{z_{0}^{n+1}}, \eta_{z_{0}^{n+1}}) \in S^{s}_{{ \cal V}^{n}}$ then 

\begin{equation} | \eta_{f_{n}^{-1}(z_{0}^{n+1})} | \geq \left( \frac{1}{\mu} \right) | \eta_{z_{0}^{n+1}} | \quad \text{for } \mu > 0 
\end{equation}
\end{assumption}

\noindent
Obviously we need to impose an additional condition in order to guarantee the existence of the Jacobian matrices $Df_{n}$ and $Df_{n}^{-1}$. From now we will consider that $f_{n},f_{n}^{-1} \in C^{1}$ for every $n \in \mathbb{Z}$ on their respective domains. Now we establish an important relationship between assumptions A2 and A3.

\begin{theorem} If nonautonomous A1 and A3 are satisfied for $0< \mu < 1-\mu_{h}\mu_{v}$ then A2 is satisfied.
\label{thm4}
\end{theorem}

\noindent
Part of the proof of this theorem is based on the following result.
\\
\begin{Lem}
Let $\lbrace f_{n},D_{n} \rbrace_{n=-\infty}^{+\infty}$ be a sequence of maps satisfying A1 and A3. \\ For every $n \in \mathbb{Z}$ and every pair of indices $i,j \in I$ we have that
\newline
\\
i) if $\overline{V}^{n+1} \subset V_{j}^{n+1}$ is a $\mu_{v}$-vertical curve, then $f_{n}^{-1}(\overline{V}^{n+1}) \cap V_{i}^{n}$ is a $\mu_{v}$-vertical curve in case $\overline{V}^{n+1} \cap H_{i}^{n+1} \not = \emptyset$.
\newline
\\
ii) if $\overline{H}^{n} \subset V_{ji}^{n}$ is a $\mu_{h}$-horizontal curve, then $f_{n}(\overline{H}^{n}) \cap H_{i}^{n+1}$ is a $\mu_{h}$-horizontal curve in case $\overline{H}^{n} \cap V_{i}^{n} \not = \emptyset$.
\label{lem5}
\end{Lem}

\begin{proof} We omit the proof of ii) as it follows the same line of reasoning as i).
\newline
\\
We consider a $\mu_{v}$-vertical curve $\overline{V}^{n+1} \subset V_{j}^{n+1}$. By definition there exist an interval $T \subset \mathbb{R}$ and a function $v:T \rightarrow \mathbb{R}$ such that $\overline{V}^{n+1}$ is the graph of $v$ and also the Lipschitz condition $|v(t_{1})-v(t_{2})| \leq \mu_{v}|t_{1}-t_{2}|$ holds for a constant $\mu_{v}>0$ and every pair of points $t_{1},t_{2} \in T$.
\newline
\\
It follows from Assumption 1 that $(f_{n}^{-1})$ is a homeomorphism over $H_{ij}^{n+1}=H_{i}^{n+1} \cap V_{j}^{n+1}$. In particular a homeomorphism over $\overline{V}^{n+1} \cap H_{i}^{n+1} \not = \emptyset$. This implies that

\begin{equation}
f_{n}^{-1} ( \overline{V}^{n+1} \cap H_{i}^{n+1} ) = f_{n}^{-1} ( \overline{V}^{n+1} ) \cap f_{n}^{-1} ( H_{i}^{n+1} ) = f_{n}^{-1} ( \overline{V}^{n+1} ) \cap V_{i}^{n} \not = \emptyset
\end{equation}

\noindent
Since the curve $\overline{V}^{n+1}$ can be parametrized by $(v(t),t)|_{t \in T}$ (take also $v \in C^{1}$) then this last subset $f_{n}^{-1} ( \overline{V}^{n+1} ) \cap V_{i}^{n}$ can also have a parametrization but over a smaller domain $T^{*} \subset T$,

\begin{equation}
\left( \begin{array}{c} x(t) \\ y(t) \end{array} \right) = f_{n}^{-1}(v(t),t) \quad \text{with} \quad t \in T^{*} \equiv \lbrace \overline{t} \in T : (v(\overline{t}),\overline{t}) \in H_{i}^{n+1} \rbrace
\end{equation}

\noindent
The image of ant tangent vector of $\overline{V}^{n+1}$ under $Df_{n}^{-1}$ has the form

\begin{equation}
\left( \begin{array}{c} \dot{x}(t) \\ \dot{y}(t) \end{array} \right) = Df_{n}^{-1}(v(t),t) \cdot \left( \begin{array}{c} \dot{v}(t) \\ 1 \end{array} \right) \quad \text{with} \quad \left( \begin{array}{c} \dot{v}(t) \\ 1 \end{array} \right) \in S^{s}_{{\cal H}^{n+1}} , \quad \forall t \in T^{*}
\end{equation}

This last relation follows directly  from the Lipschitz condition:

\begin{equation}
|\dot{v}(t)| \leq \begin{array}{c} \lim \sup \\ \text{{\small $\epsilon \rightarrow 0$}} \\ \text{{\small $(t+\epsilon ) \in T^{*}$}} \end{array} \left| \frac{v(t+\epsilon )-v(t)}{\epsilon } \right| \leq \begin{array}{c} \lim \sup \\ \text{{\small $t_{1} \in T^{*}$}} \\ \text{{\small $t_{1} \not = t$}} \end{array} \frac{|v(t_{1})-v(t)|}{|t_{1}-t|} \leq \begin{array}{c} \lim \sup \\ \text{{\small $t_{1} \in T^{*}$}} \\ \text{{\small $t_{1} \not = t$}} \end{array} \frac{\mu_{v}|t_{1}-t|}{|t_{1}-t|} = \mu_{v}
\end{equation}

By applying Assumption 3 we obtain that the tangent vectors belong to $S^{s}_{{\cal V}^{n}}$,

\begin{equation}
|\dot{x}(t)| \hspace{0.1cm} \leq \hspace{0.1cm} \mu_{v} \cdot |\dot{y}(t)| \quad , \quad \forall t \in T^{*}
\end{equation}

\noindent
Moreover, as we also assume that $(f_{n}^{-1})\in C^{1}$, any tangent vector

$$\left( \begin{array}{c} \dot{x}(t) \\ \dot{y}(t) \end{array} \right) = Df_{n}^{-1}(v(t),t) \cdot \left( \begin{array}{c} \dot{v}(t) \\ 1 \end{array} \right) \text{ cannot be equal to } \left( \begin{array}{c} 0 \\ 0 \end{array} \right) \text{ at any point } t \in T^{*}.$$

\noindent
From these two relations it follows that $\dot{y}(t)$ cannot change its sign in the entire domain $T^{*}$. Consequently for every pair of points $(x_{1},y_{1}),(x_{2},y_{2}) \in f_{n}^{-1} ( \overline{V}^{n+1} ) \cap V_{i}^{n}$ there exist $t_{1},t_{2} \in T^{*}$ such that $(x_{k},y_{k})=(x(t_{k},y(t_{k}))$, $(k=1,2)$ and we have the inequality

$$|x_{1}-x_{2}|=|x(t_{1})-x(t_{2})|= \left| \int_{t_{2}}^{t_{1}} \dot{x}(t)dt \right| \leq \int_{t_{2}}^{t_{1}} |\dot{x}(t)|dt \leq \mu_{v}\int_{t_{2}}^{t_{1}} |\dot{y}(t)|dt =$$

\begin{equation} = \mu_{v} \left| \int_{t_{2}}^{t_{1}} \dot{y}(t)dt \right| = \mu_{v} |y(t_{1})-y(t_{2})| = \mu_{v}|y_{1}-y_{2}|
\end{equation}

\noindent
and this result implies  that $f_{n}^{-1} ( \overline{V}^{n+1} ) \cap V_{i}^{n}$ is a $\mu_{v}$-vertical curve. 
\end{proof}

\begin{proof}[Proof of Theorem \ref{thm4}] The theorem will be proved by verifying the following steps.
\newline
\\
\underline{Step 1:} Let $\overline{V}^{n+1} \subset V_{j}^{n+1}$ be a $\mu_{v}$-vertical curve. Then $f_{n}^{-1}(\overline{V}^{n+1})\cap V_{i}^{n}$ is a $\mu_{v}$-vertical curve for every $i \in I$ such that $\overline{V}^{n+1}\cap H_{i}^{n+1} \not = \emptyset$.
\newline
\\
\underline{Step 2:} Let $V^{n+1}$ be a $\mu_{v}$-vertical strip which intersects $V_{j}^{n+1}$ fully. Then $f_{n}^{-1}(V^{n+1})\cap V_{i}^{n}$ is a $\mu_{v}$-vertical strip that intersects $V_{i}^{n}$ fully for every $i \in I$ such that $V^{n+1}\cap H_{i}^{n+1} \not = \emptyset$.
\newline
\\
\underline{Step 3:} Show that $d(\widetilde{V}_{i}^{n}) \leq (\mu /( 1-\mu_{h}\mu_{v})) \cdot d(V^{n+1})$ for $\widetilde{V}_{i}^{n}=f_{n}^{-1}(V^{n+1})\cap V_{i}^{n}$.
\newline
\\
We omit the part of the proof dealing with horizontal strips since it follows from the same reasoning used to prove the part concerning vertical strips.
\newline
\newline
\\
We begin with Step 1. Let $\overline{V}^{n+1} \subset V_{j}^{n+1}$ be a $\mu_{v}$-vertical curve. For each $i \in I$ such that $\overline{V}^{n+1}\cap H_{i}^{n+1} \not = \emptyset$, by applying A1 we obtain that $H_{ij}^{n+1}=V_{j}^{n+1} \cap H_{i}^{n+1} \not = \emptyset$ is a $\mu_{h}$-horizontal strip contained in $V_{j}^{n+1}$. Since implicitly we are taking $\overline{V}^{n+1}$ as one of the two components of the vertical boundary of a vertical strip $V^{n+1}$ intersecting $V_{j}^{n+1}$ fully, the curve $\overline{V}^{n+1}$ intersects $\partial_{h}H_{i}^{n+1}$ in exactly two points.
\newline
\\
Also because $f_{n}^{-1}$ maps the horizontal boundaries of each subset $H_{ij}^{n+1}(= H_{i}^{n+1} \cap V_{j}^{n+1})$ onto the horizontal boundaries of $V_{i}^{n}$, we have that $f_{n}^{-1}(\overline{V}^{n+1})\cap V_{i}^{n}$ is a curve linking the two horizontal boundaries of $V_{i}^{n}$. Finally if we apply Lemma \ref{lem5} to this curve it follows that $f_{n}^{-1}(\overline{V}^{n+1}) \cap V_{i}^{n}$ is also a $\mu_{v}$-vertical curve.
\newline
\\

To prove Step 2 we apply Step 1 to the $\mu_{v}$-vertical boundaries of the $\mu_{v}$-vertical strip $V^{n+1}$ which intersects $V_{j}^{n+1}$ fully. It then follows that $f_{n}^{-1}(V^{n+1}) \cap V_{i}^{n}$ is also a $\mu_{v}$-vertical strip for every $i \in I$ such that $V^{n+1} \cap H_{i}^{n+1} \not = \emptyset$. Moreover this last strip intersects each $V_{i}^{n}$ fully because of the geometric considerations in Step 1.
\newline
\\
For proving Step 3 first we need to fix an iteration $n \in \mathbb{Z}$ and an index $i \in I$. The width of each $\mu_{v}$-vertical strip $\widetilde{V}_{i}^{n}$ will be the distance between two points $p_{0},p_{1} \in \widetilde{V}_{i}^{n}$ with the same $y$-component and located in separate vertical boundaries, $d(\widetilde{V}_{i}^{n})= | p_{1}-p_{0} |$.
\begin{figure}[h!]
\centering
\includegraphics[scale=0.4]{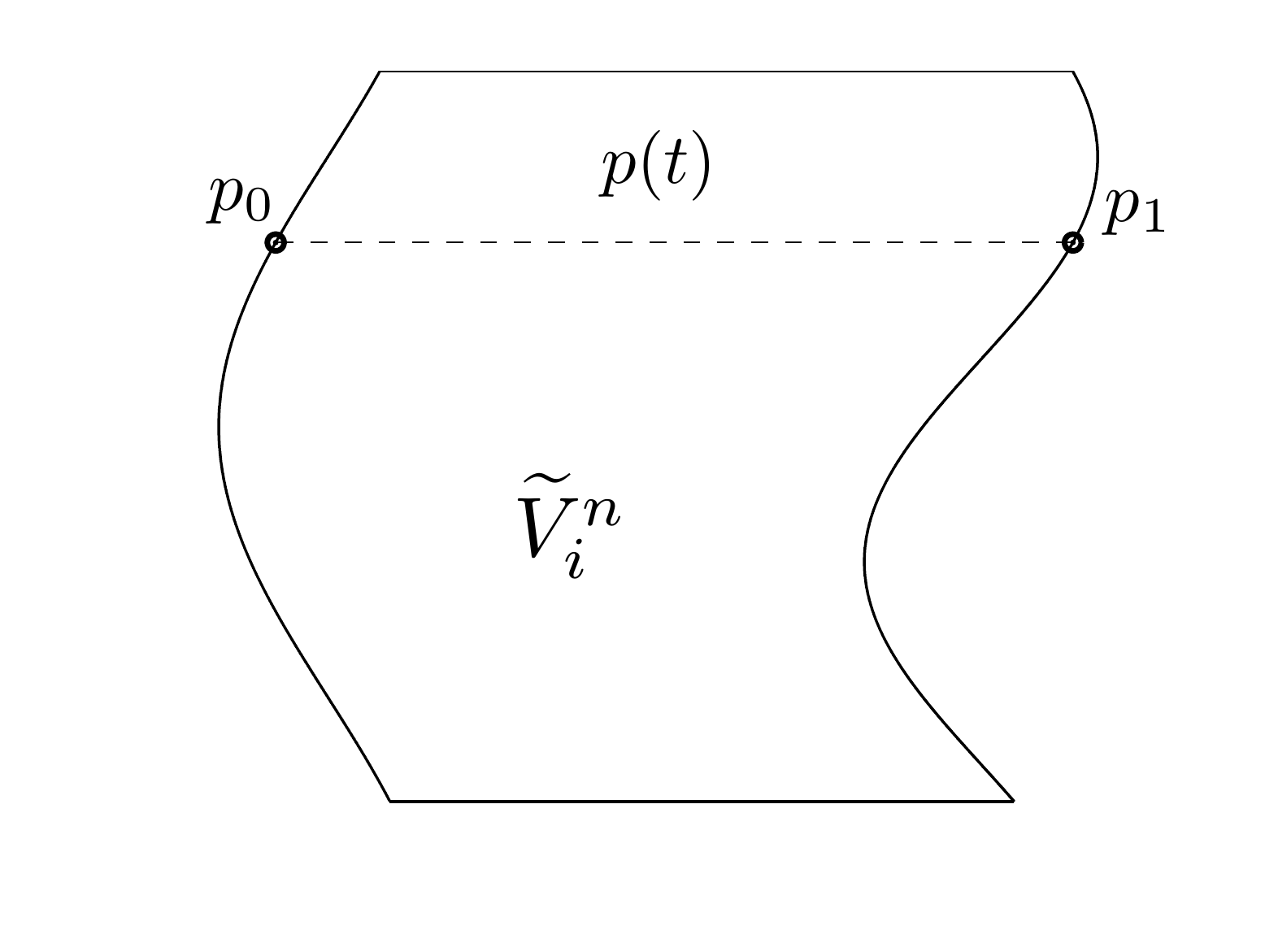}
\caption{The segment $p(t)=tp_{1}+(1-t)p_{0}$, $t \in [0,1]$, which represents the maximum amplitude of $\widetilde{V}_{i}^{n}$ is obviously a $\mu_{h}$-horizontal curve.}
\end{figure} 
\label{fig4}
\noindent
\\
By taking segment $p(t)$ considered in Figure \ref{fig4}, $\dot{p}(t)=p_{1}-p_{0}$ is a vector with its $y$-component equal to zero. Therefore $\dot{p}(t) \in S^{u}_{{ \cal V}^{n}}$, $\forall t \in [0,1]$. Now we have that the curve $f_{n}(p(t)) \equiv z(t)=(x(t),y(t))$ located in $D_{n+1}$ is a $\mu_{h}$-horizontal curve because of the second part of Lemma \ref{lem5},

\begin{equation}\text{Moreover A3 states that} \quad \dot{z}(t)=D(f_{n}(p(t)))=Df_{n}(p(t))\cdot \dot{p}(t) \in S^{u}_{{ \cal H}^{n+1}}\end{equation}
\\
Furthermore since the graph of $z(t)=(x(t),y(t))$ is a $\mu_{h}$-horizontal curve, we obtain that
\begin{equation}|y(1)-y(0)| \leq \mu_{h} |x(1)-x(0)| \quad \rightarrow \quad |y_{1}-y_{0}| \leq \mu_{h} |x_{1}-x_{0}|\end{equation}
by denoting $(x_{i},y_{i}) \equiv (x(i),y(i))=z(i)=f_{n}(p(i))=f_{n}(p_{i})$ for $i=0,1$.

\begin{figure}[h!]
\centering
\includegraphics[scale=0.5]{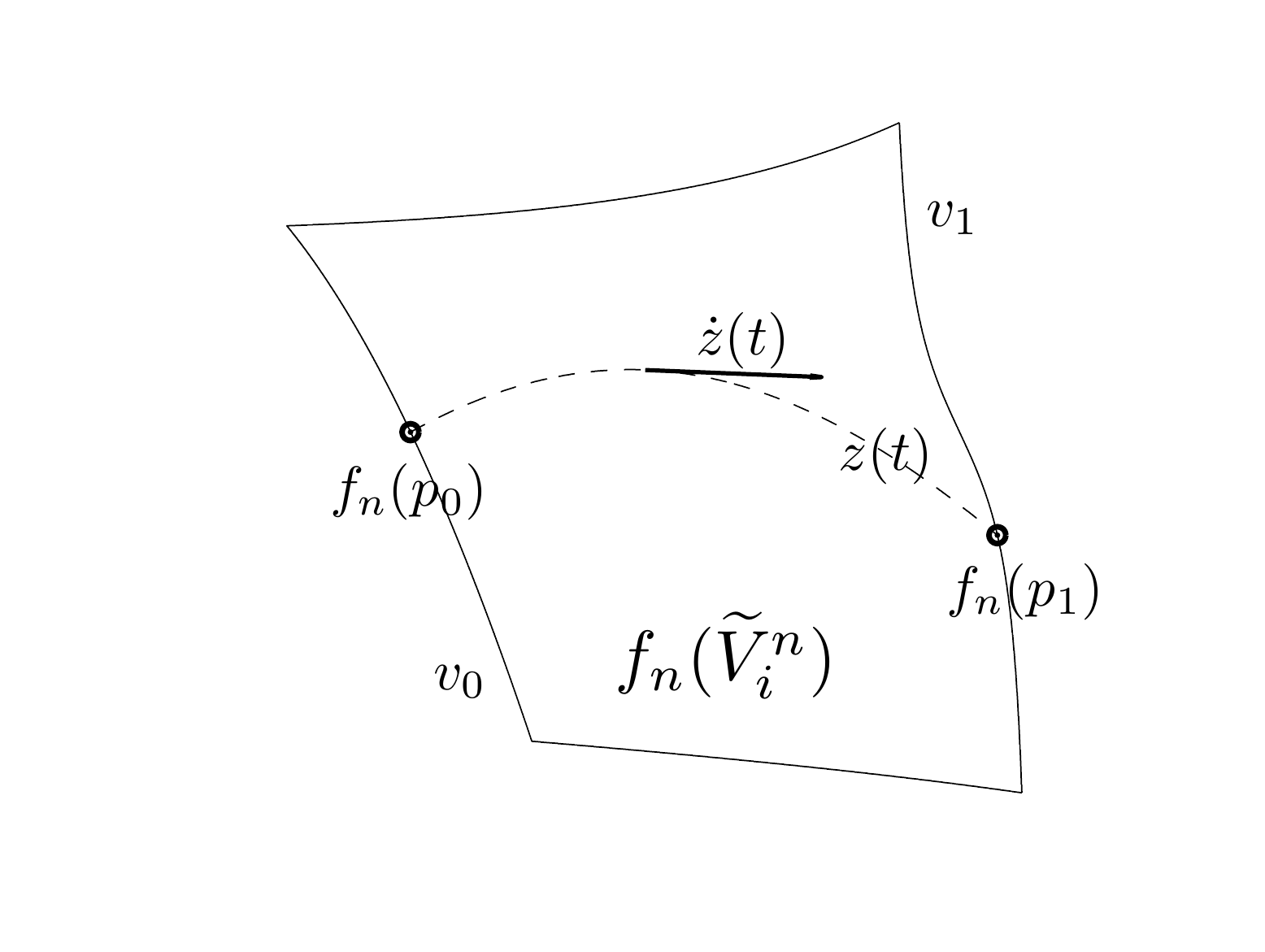}
\caption{$f_{n}(p_{0})$ and $f_{n}(p_{1})$ are on the graphs of two distinct $\mu_{v}$-vertical curves, which we denote by $v_{0}$ and $v_{1}$ respectively.}
\label{fig5}
\end{figure}
\noindent
\\
Using this last fact and also the geometric considerations in Figure \ref{fig5}, it follows that

$$|x_{1}-x_{0}| = |v_{1}(y_{1})-v_{0}(y_{0})| \leq |v_{1}(y_{1})-v_{1}(y_{0})|+|v_{1}(y_{0})-v_{0}(y_{0})| \leq$$

$$\leq \mu_{v}|y_{1}-y_{0}| + d(V^{n+1}) \leq \mu_{v} \mu_{h} |x_{1}-x_{0}|+ d(V^{n+1}) \rightarrow$$

\begin{equation}\rightarrow \quad |x_{1}-x_{0}| \leq \frac{d(V^{n+1})}{(1-\mu_{h}\mu_{v})}\end{equation}

Also as a result  of the last part of Assumption 3 there exists a positive constant $\mu$, which we impose to be $\mu < 1-\mu_{h}\mu_{v}$, such that

$$|\dot{x}(t)| \geq \left( \frac{1}{\mu} \right) |\dot{p}(t)| = \left( \frac{1}{\mu} \right) |p_{1}-p_{0}| \quad \text{and then}$$

\begin{equation}d(\widetilde{V}_{i}^{n})=|p_{1}-p_{0}| \leq \mu \int_{0}^{1} |\dot{x}(t)|dt = \mu \left| \int_{0}^{1} \dot{x}(t)dt \right| = \mu |x_{1}-x_{0}|\end{equation}

Note that the two expressions containing the integrals are equal since $\dot{x}(t)$ does not change its sign at any point. This is due to the fact that the graph of $z(t)=(x(t),y(t))$ is a $\mu_{h}$-horizontal curve.
\newline
\\
Finally we arrive at the result,
\begin{equation}d(\widetilde{V}_{i}^{n})=|p_{1}-p_{0}| \leq \mu |x_{1}-x_{0}| \leq \frac{\mu}{(1-\mu_{h}\mu_{v})}d(V^{n+1})\end{equation}
$$\text{and} \quad \nu_{v}=\frac{\mu}{(1-\mu_{h}\mu_{v})}<1 \quad \text{will be the required constant for Assumption 2.}$$
\end{proof}

\section{Nonautonomous H\'{e}non map}
\label{sec:nahmap}

We have now developed the necessary tools for proving the existence of a  chaotic invariant set for  the nonautonomous H\'{e}non map. Recall our general notation for nonautonomous dynamics (a sequence of maps defined on a sequence of domains), $\lbrace f_{n},D_{n} \rbrace_{n=-\infty}^{+\infty}$.
\newline
\\
We will construct  domains $D_{n}$ for the nonautonomous H\'enon map, each of them containing an associated pair of horizontal strips and another of vertical strips. Moreover, the transition matrices will be shown to be identical for each map $f_n$, with  $A= \left( \begin{array}{cc} 1 & 1 \\ 1 & 1 \end{array} \right)$ for each iteration $n$.

Recall that the autonomous  H\'{e}non map has the form:

\begin{equation} H(x,y)=(A+By-x^{2},x) , \quad \text{with inverse function} \quad  H^{-1}(x,y)=(y, (x-A+y^{2})/B) \end{equation}

\noindent
Following \cite{Dev79}, sufficient conditions for the existence of a chaotic invariant set in the autonomous context can be proven when the parameters satisfy the following inequalities:

\begin{equation}
A > A_{2}=\frac{(5+2\sqrt{5})(1+|B|)^{2}}{4} \quad , \quad A_{2}=5+2\sqrt{5} \approx 9.47 \quad \text{in case} \quad B= \pm 1
\end{equation}
\\Note that when  $B=-1$ the map is orientation-preserving and area-preserving. For our version of the nonautonomous H\'{e}non map,  we will take $B=-1$ for the sequence of maps $\lbrace f_{n},D_{n} \rbrace_{n=-\infty}^{+\infty}$  in order to retain these  properties, but we will allow  $A$ to vary for each iteration $n$. Therefore, we will take:

\begin{equation} f_{n}(x,y)=(A(n)-y-x^{2},x) \quad , \quad f_{n}^{-1}(x,y)=(y,A(n)-x-y^{2}) \end{equation}

where 

\begin{equation} A(n)=9.5 + \epsilon \cdot \cos (n) \quad \text{with} \quad \epsilon =0.1. \end{equation}

This choice is motivated by the fact that $A_{2}=5+2\sqrt{5} \approx 9.47$ is the minimum threshold for parameter $A$ for which the autonomous H\'{e}non map satisfies the autonomous versions of Assumptions 1 and 3 of the Conley-Moser conditions.

In the following we will prove that the  nonautonomous H\'{e}non map satisfies the conditions described in Theorem \ref{mainthm}. In particular, we will prove the following theorem.
\\
\begin{theorem}
If $A^{*} \geq 9.5$ then the nonautonomous H\'{e}non map $f_{n}=(A(n)-y-x^{2},x)$ with $A(n)=A^{*} + \epsilon \cdot \cos (n)$, $\epsilon =0.1$ has  a nonautonomous chaotic invariant set in $\mathbb{R}^{2}$.
\end{theorem}

\begin{proof} We carry out the proof for the specific case where $A_{0}=9.5$. The case for $A_{0}>9.5$  follows similar reasoning as for the case $A_{0}=9.5$, with the main difference being that  some values in the inequalities appearing when checking Assumption 3 must be changed. We begin with the first Conley-Moser condition.
\newline
\\
\underline{Assumption 1}. The domain $D_{n}$ on which each function $f_{n}$ will be defined is the square

\begin{equation} D_{n}=D=[-R,R] \times [-R,R] \quad \quad \text{with} \quad \quad R=\sup_{n \in \mathbb{Z}} R(n)= 1+\sqrt{1+A(0)} \approx 4.25 \end{equation}
analogously to the domain considered for the autonomous H\'enon map.
\newline
\\
The horizontal strips and the vertical strips associated to any iteration $n \in \mathbb{Z}$ will be taken as

\begin{equation} 
D_{H}^{n+1} \equiv f_{n}(D) \cap D \quad , \quad D_{V}^{n} \equiv f_{n}^{-1}(D) \cap D \end{equation}
\\
and since $f_{n}$ is a homeomorphism we also note that vertical strips ``move'' to  horizontal strips in forward iteration,
\begin{equation} f_{n}(D_{V}^{n})=f_{n} \left( f_{n}^{-1}(D) \cap D \right) = (f_{n} \circ f_{n}^{-1})(D) \cap f_{n}(D) = D_{H}^{n+1} \end{equation}
\\
Moreover the index $I$ indicating the number of strips in either $D_{H}^{n}$ or $D_{V}^{n}$ is $I=\lbrace 1,2 \rbrace$ and the strips are defined by
\\
$$H_{1}^{n+1} \equiv f_{n}(D) \cap \left( [-R,R] \times [0,R] \right) \quad , \quad H_{2}^{n+1} \equiv f_{n}(D) \cap \left( [-R,R] \times [-R,0] \right)$$
\begin{equation} V_{1}^{n} \equiv f_{n}^{-1}(D) \cap \left( [0,R] \times [-R,R] \right) \quad , \quad V_{2}^{n} \equiv f_{n}^{-1}(D) \cap \left( [-R,0] \times [-R,R] \right) \end{equation}
\\
These are determined by the images of $D$ with respect to $f_{n}$ and $f_{n}^{-1}$ for every $n \in \mathbb{Z}$. They result easy to compute. Let
\\
$$L_{1}= \lbrace (x,y) \in D \text{ } | \text{ } y=R \rbrace \quad , \quad L_{2}= \lbrace (x,y) \in D \text{ } | \text{ } y=-R \rbrace$$
$$L_{3}= \lbrace (x,y) \in D \text{ } | \text{ } x=R \rbrace \quad , \quad L_{4}= \lbrace (x,y) \in D \text{ } | \text{ } x=-R \rbrace$$
\\
the segments which conform the boundary of $D$. Their images with respect $f_{n}$ and $f_{n}^{-1}$ are either another segment or a parabola, and as $f_{n}$ is a homeomorphism, both $f_{n}(D)$ and $f_{n}^{-1}(D)$ are two strips with a parabolic form.

\begin{figure}[h!]
\centering
\includegraphics[scale=0.40]{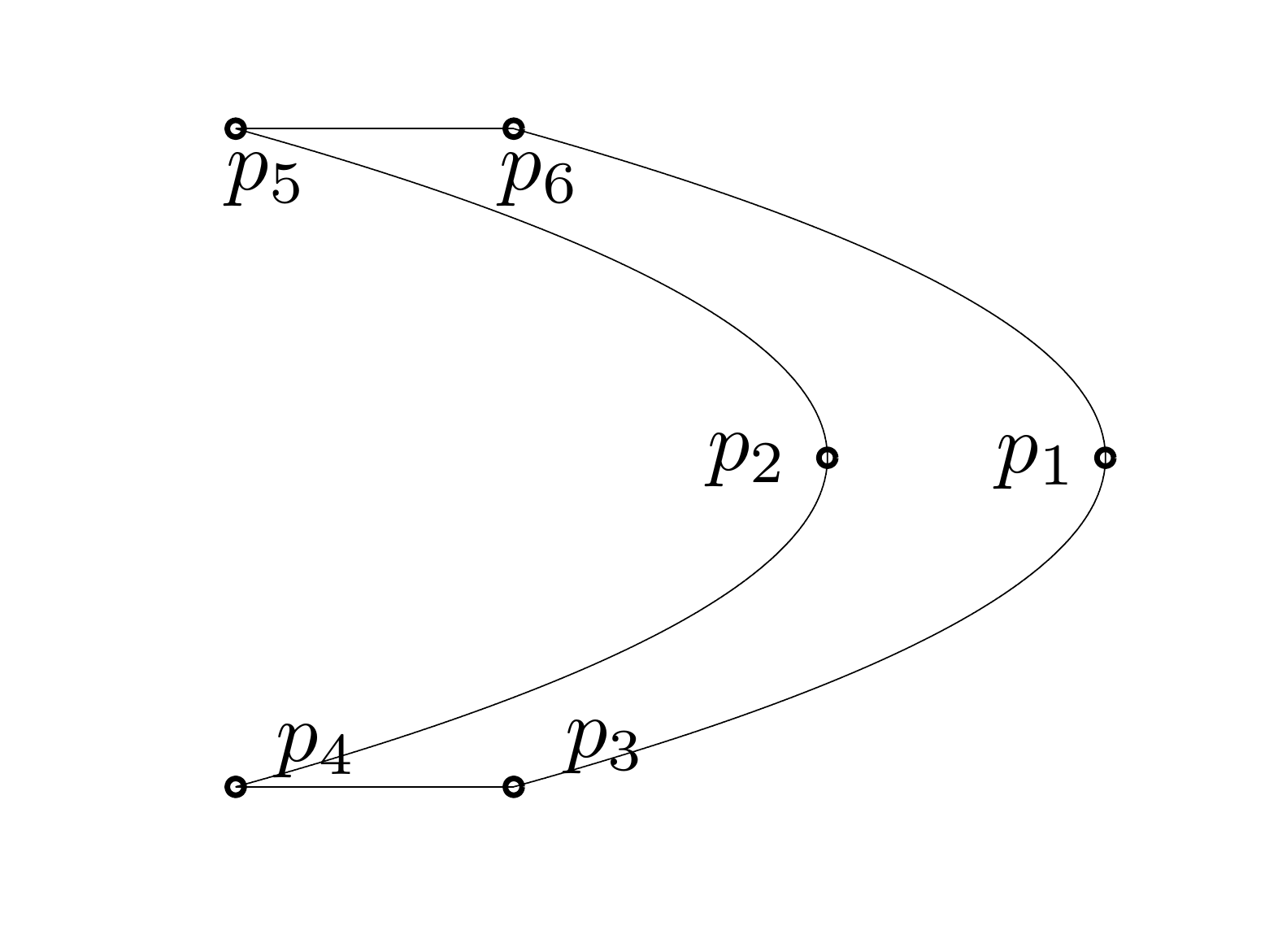} \quad \includegraphics[scale=0.40]{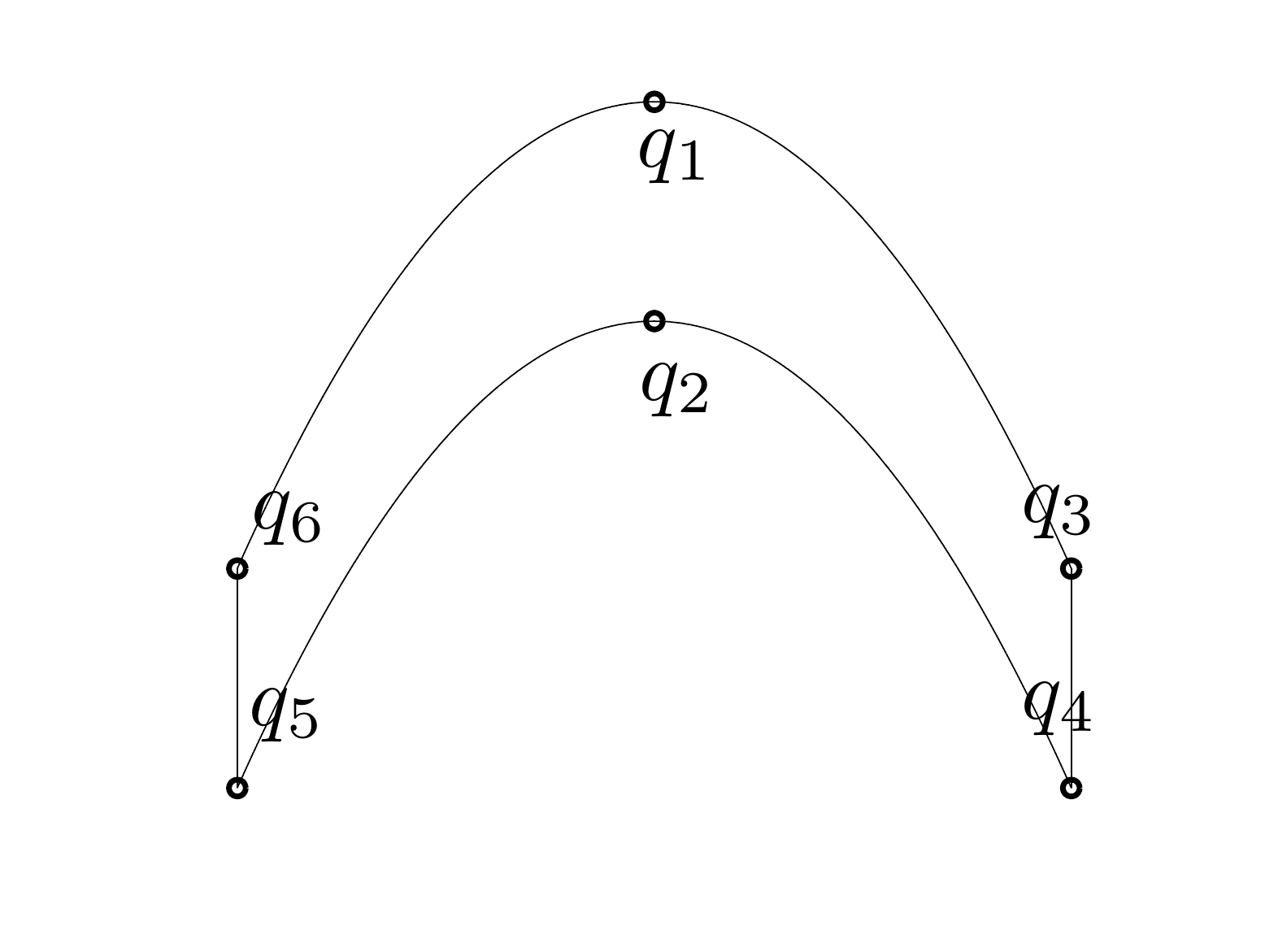}
\caption{$f_{n}(D)$ and $f_{n}^{-1}(D)$ take these two shapes, respectively, for any given $n \in \mathbb{Z}$. The set of points $p_{1},p_{2},p_{3},p_{4},p_{5},p_{6}$ and $q_{1},q_{2},q_{3},q_{4},q_{5},q_{6}$ determine the height and length of both geometric forms.}
\label{fig6}
\end{figure}

\noindent
\\
The key points of $f_{n}(D)$ and $f_{n}^{-1}(D)$ shown in Figure \ref{fig6} have the following coordinates,

$$p_{1} \equiv (A(n)+R,0) \quad , \quad p_{2} \equiv (A(n)-R,0) \quad , \quad p_{3} \equiv (A(n)+R-R^{2},-R)$$

$$p_{4} \equiv (A(n)-R-R^{2},-R) \quad , \quad p_{5} \equiv (A(n)-R-R^{2},R) \quad , \quad p_{6} \equiv (A(n)+R-R^{2},R)$$

$$q_{1} \equiv (0,A(n)+R) \quad , \quad q_{2} \equiv (0,A(n)-R) \quad , \quad q_{3} \equiv (R,A(n)+R-R^{2})$$

$$q_{4} \equiv (R,A(n)-R-R^{2}) \quad , \quad q_{5} \equiv (-R,A(n)-R-R^{2}) \quad , \quad q_{6} \equiv (-R,A(n)+R-R^{2})$$

The coordinates of the these points satisfy $A(n) > 2R$, $\forall n \in \mathbb{Z}$ and

\begin{equation}
A(n)+R-R^{2}=A(n)+1+\sqrt{1+A(0)}-(1+\sqrt{1+A(0)})^{2}=(A(n)-A(0))-R \leq -R
\end{equation}

\noindent
with strict inequality when $n \not = 0$. Only in case $n=0$, the points $p_{6},p_{3}=q_{6},q_{3}$ are inside the domain $D$ and actually these are three vertices of the square $D$. In any case, it follows that the points $p_{1},p_{2},p_{4},p_{5}$ and $q_{1},q_{2},q_{4},q_{5}$ do not belong to $D$ for any $n$.

We denote the arguments of the parabolic curves by $X$ and $Y$, so that their equations take the form:

\begin{equation} Y=\sqrt{A(n)-R-X} \quad \text{in the horizontal case}, \end{equation}

\begin{equation} X=\sqrt{A(n)-R-Y} \quad \text{in the vertical case}. \end{equation}
\\
With this notation  the absolute value of the derivatives of these functions take the form:

\begin{equation} \left| \frac{dY}{dX} \right| = \frac{1}{2\sqrt{A(n)-R-X}} \quad , \quad \left| \frac{dX}{dY} \right| = \frac{1}{2\sqrt{A(n)-R-Y}} \end{equation}
\\
and their maximum values are assumed when $X=R$ and $Y=R$, respectively. Therefore we have:

$$\frac{1}{2\sqrt{A(n)-2R}} \leq \frac{1}{2\sqrt{9.5 + \epsilon \cos (n) - 2R}} \approx \frac{1}{2 \sqrt{9.5 - 0.1 - 8.5 }} = \frac{1}{2 \sqrt{0.9}} \approx 0.527$$
\newline
\\
Due to reasons explained later, for convenience we will choose the thresholds $\mu_{h}=\mu_{v}=0.615$ for the maximum values that the slopes of the horizontal and vertical boundaries can assume, respectively. Using this fact, one can conclude that $H_{i}^{n+1}$ is a $\mu_{h}$-horizontal strip and $V_{i}^{n}$ a $\mu_{v}$-vertical strip for every $i \in I$ and $n \in \mathbb{Z}$. Moreover, $\mu_{h} \cdot \mu_{v} = (0.615)^{2} = 0.378225 < 1$.  This proves part of Assumption 1.
\newline
\\
Furthermore, for any $i,j \in I$ and $n \in \mathbb{Z}$ we have that the horizontal boundaries of $f_{n}(V_{i}^{n})=H_{i}^{n+1}$ are two $\mu_{h}$-horizontal curves which link the left and right sides of the square $D$. Since the two $\mu_{v}$-vertical curves bounding $\partial_{v}V_{j}^{n+1}$ link the upper and the bottom sides of $D_{n+1}$, both boundaries $\partial_{h}f_{n}(V_{i}^{n})$ and $\partial_{v}V_{j}^{n+1}$ intersect in four different points. From this fact it follows that $H_{ij}^{n+1}=f_{n}(V_{i}^{n}) \cap V_{j}^{n+1}$ is a $\mu_{h}$-horizontal strip contained in $V_{j}^{n+1}$.
\newline
\\
Since $f_{n}$ is a homeomorphism for every $n \in \mathbb{Z}$,

$$ f_{n} \quad \text{maps} \quad V_{ji}^{n}=f_{n}^{-1}(V_{j}^{n+1}) \cap V_{i}^{n} = f_{n}^{-1} \left( f_{n}(V_{i}^{n}) \cap V_{j}^{n+1} \right) \quad \text{onto} \quad H_{ij}^{n+1} \quad \text{and} $$

\begin{equation}f_{n}^{-1} \left( \partial_{h}H_{ij}^{n+1} \right) \subset \partial_{h}V_{i}^{n} \quad \text{because by construction} \quad \partial_{h} H_{ij}^{n+1} \subset \partial_{h}H_{i}^{n+1} \end{equation}

This can be checked by an easy computation.
\newline
\\
Therefore the  nonautonomous H\'{e}non map satisfies Assumption 1.
\newline
\\
\underline{Assumption 3}. To begin our verification that A3 is also satisfied, we need to recall  the notation for several concepts developed earlier:

\begin{equation} Df_{n}(x,y)= \left( \begin{array}{cc} -2x & -1 \\ 1 & 0 \end{array} \right) \quad , \quad Df_{n}^{-1}(x,y)= \left( \begin{array}{cc} 0 & 1 \\ -1 & -2y \end{array} \right) \end{equation}

$$ S^{u}_{{ \cal K}} = \lbrace (\xi_{z},\eta_{z}) \in \mathbb{R}^{2} \text{ } | \text{ } | \eta_{z} | \leq \mu_{h} | \xi_{z} | \text{, } z \in { \cal K} \rbrace $$

$$ S^{s}_{{ \cal K}} = \lbrace (\xi_{z},\eta_{z}) \in \mathbb{R}^{2} \text{ } | \text{ } | \xi_{z} | \leq \mu_{v} | \eta_{z} | \text{, } z \in { \cal K} \rbrace $$
\\
with ${\cal K}$ being either ${\cal V}^{n}$ or ${\cal H}^{n+1}$.
\newline
\\
Now given any point $z_{0}=(x_{0},y_{0}) \in { \cal H}^{n+1}$ and any $(\xi_{z_{0}} , \eta_{z_{0}}) \in S^{s}_{{ \cal H}^{n+1}}$ (which by definition, $|\xi_{z_{0}}| \leq \mu_{v} |\eta_{z_{0}}|$) we have that

$$Df_{n}^{-1}(z_{0}) \cdot (\xi_{z_{0}},\eta_{z_{0}})= \left( \begin{array}{cc} 0 & 1 \\ -1 & -2y_{0} \end{array} \right) \cdot \left( \begin{array}{c} \xi_{z_{0}} \\ \eta_{z_{0}} \end{array} \right) = \left( \begin{array}{c} \eta_{z_{0}} \\ -\xi_{z_{0}}-2y_{0}\eta_{z_{0}} \end{array} \right)$$

\noindent
belongs to $S^{s}_{{ \cal V}^{n}}$ if and only if the inequality

\begin{equation} 
|\eta_{z_{0}}| \leq \mu_{v} \cdot |\xi_{z_{0}}+2y_{0}\eta_{z_{0}}| \quad \quad \text{holds.} 
\label{61}
\end{equation}
\\
Since it is also true that

$$ \mu_{v} \cdot |\xi_{z_{0}}+2y_{0}\eta_{z_{0}}| \geq  \mu_{v} \cdot \left[ 2|y_{0}||\eta_{z_{0}}|-|\xi_{z_{0}}| \right] \geq$$

\begin{equation} \geq \mu_{v} \cdot \left[ 2|y_{0}||\eta_{z_{0}}|-\mu_{v}|\eta_{z_{0}}| \right] = \left( 2|y_{0}|\mu_{v}-\mu_{v}^{2} \right) |\eta_{z_{0}}| \end{equation}

In case $\left( 2|y_{0}|\mu_{v}-\mu_{v}^{2} \right) \geq 1$, the previous inequality \eqref{61} will hold.
\newline
\\
To verify this we need to check if

\begin{equation} |y_{0}| \geq \frac{1}{2} \left( \mu_{v}+\frac{1}{\mu_{v}} \right)=1.1205 \quad \text{for any } z_{0}=(x_{0},y_{0}) \in { \cal H}^{n+1}. \end{equation}
\\
At this point we give a geometrical argument.
\newline
\\
$\bullet$ The horizontal lines $\lbrace Y= \pm 1.1205 \rbrace$ cut the parabola $\lbrace X=A(n)-R-Y^{2} \rbrace$ at two points:
$$(x_{1},y_{1})=(8.5+ \epsilon \cos (n) - \sqrt{10.6}-1.2555,1.1205) \quad \text{and}$$
$$(x_{2},y_{2})=(8.5+ \epsilon \cos (n) - \sqrt{10.6}-1.2555,-1.1205)$$
\\
$\bullet$ The horizontal lines $\lbrace Y= \pm 1.1205 \rbrace$ cut the parabola $\lbrace Y=A(n+1)+R-X^{2} \rbrace$ at two points with a positive $x$-component:
$$(\bar{x}_{1},\bar{y}_{1})=(\sqrt{10.5+ \epsilon \cos (n+1) + \sqrt{10.6}-1.1205},1.1205) \quad \text{and}$$
$$(\bar{x}_{2},\bar{y}_{2})=(\sqrt{10.5+ \epsilon \cos (n+1) + \sqrt{10.6}+1.1205},-1.1205)$$
\\
From here we have that

$$\bar{x}_{1}<\bar{x}_{2}=\sqrt{10.5+ \epsilon \cos (n+1) + \sqrt{10.6}+1.1205} \leq \sqrt{10.6 + \sqrt{10.6}+1.1205}=$$

$$\sqrt{14.9758}=3.8699 < 3.8887=8.5-0.1-\sqrt{10.6}-1.2555 \leq $$

\begin{equation}\leq 8.5 + \epsilon \cos (n) - \sqrt{10.6}-1.2555=x_{2}=x_{1}<4.25<R \quad \quad \forall n \in \mathbb{Z} \end{equation}

The inequalities $\bar{x}_{1} < \bar{x}_{2} < x_{2} = x_{1}$ (note that $\bar{x}_{1} < \bar{x}_{2}$ is trivial due to the definitions) also hold for every parameter $A(n) = A^{*} + \epsilon \cos (n)$ satisfying $A^{*} \geq 9.5$ and $\epsilon = 0.1$. The reason comes from comparing the derivatives of $\bar{x}_{2}$ and $x_{2}$ with respect to $A^{*}$:

\begin{equation} \bar{x}_{2} = \sqrt{A^{*} + \epsilon \cos (n+1) +1 + \sqrt{1 + A^{*} + \epsilon \cos (n)} + 1.1205} \geq \sqrt{10.4 + \sqrt{10.4} + 1.1205} \approx 3.84 \end{equation}

\begin{equation} x_{2} = A^{*} + \epsilon \cos (n) -1 -\sqrt{1+A^{*}+ \epsilon \cos (n)} -1.1205 \end{equation}

$$\frac{d \bar{x}_{2}}{dA^{*}} = \frac{1}{2 \bar{x}_{2}} \cdot \left( 1 + \frac{1}{2\sqrt{1 + A^{*} + \epsilon \cos (n)}} \right) \leq \frac{1}{2 \bar{x}_{2}} \cdot \left( 1 + \frac{1}{2\sqrt{1 + A^{*} - \epsilon }} \right) \leq$$

\begin{equation} \leq \frac{1}{2 \bar{x}_{2}} \cdot \left( 1 + \frac{1}{2\sqrt{10.4}} \right) \leq \frac{1}{2 \cdot 3.84} \cdot \left( 1 + \frac{1}{2\sqrt{10.4}} \right) \approx 0.1504 \end{equation}

\begin{equation} \frac{d x_{2}}{dA^{*}} = 1- \frac{1}{2 \sqrt{1 + A^{*} + \epsilon \cos (n)}} \geq 1 - \frac{1}{2 \sqrt{1 + A^{*} - \epsilon}} \geq 1 - \frac{1}{2 \sqrt{10.4}} \approx 0.8450 \end{equation}

Clearly $\frac{d \bar{x}_{2}}{dA^{*}} < \frac{d x_{2}}{dA^{*}}$ for every $A^{*} \geq 9.5$. It follows that $\bar{x}_{2} < x_{2}$ for $A^{*} \geq 9.5$.

This setup shows that every point $z_{0}=(x_{0},y_{0}) \in { \cal H}^{n+1}$ satisfies that $|y_{0}| > 1.1205=\frac{1}{2} \left( \mu_{v}+\frac{1}{\mu_{v}} \right)$, since the four areas composing ${ \cal H}^{n+1}$ are either beneath the line $\lbrace Y=-1.1205 \rbrace$ or above the line $\lbrace Y=1.1205 \rbrace$, as can be observed in Fig. \ref{fig7}.
\\
\begin{figure}[h!]
\centering
\includegraphics[scale=0.7]{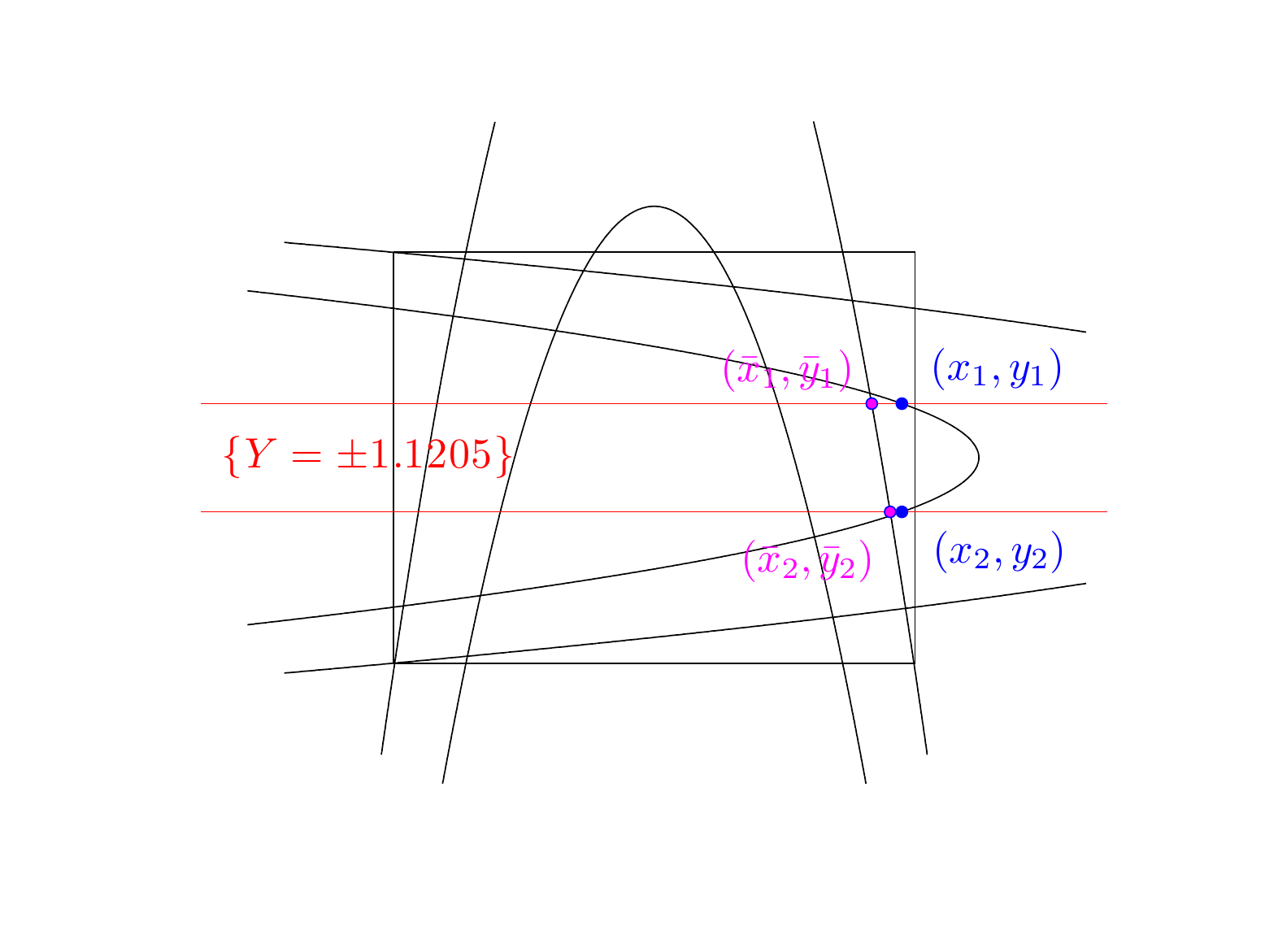}
\caption{The four areas composing ${ \cal H}^{n+1}$ are those bounded by the four parabolic strips (two horizontal and two vertical) contained in the square domain $D_{n+1}$.}
\label{fig7}
\end{figure}

\noindent
Since $z_{0} \in { \cal H}^{n+1}$ is an arbitrary point, the inclusion $Df^{-1}_{n}(S^{s}_{{ \cal H}^{n+1}})\subset S^{s}_{{ \cal V}^{n}}$ is  proven.
\newline
\\
For the second inclusion $Df_{n}(S^{u}_{{ \cal V}^{n}})\subset S^{u}_{{ \cal H}^{n+1}}$ we focus on the fact that ${ \cal V}^{n}=f_{n}^{-1}({ \cal H}^{n+1})$ and since $f_{n}^{-1}(x,y)=(y,A(n)-x-y^{2})$ transforms the $y$-components of the points of ${ \cal H}^{n+1}$ into the $x$-components of the points of ${ \cal V}^{n}$, it immediately follows that
\\
$$ |x_{0}|> \frac{1}{2} \left( \mu_{v}+\frac{1}{\mu_{v}} \right)=\frac{1}{2} \left( \mu_{h}+\frac{1}{\mu_{h}} \right) $$
\vspace{0.1cm}
\begin{equation}\text{for every point } z_{0}=(x_{0},y_{0}) \in { \cal V}^{n}. \end{equation}
\\
As in the previous case this inequality allows us to prove the inclusion,

$$Df_{n}(z_{0}) \cdot (\xi_{z_{0}},\eta_{z_{0}})= \left( \begin{array}{cc} -2x_{0} & -1 \\ 1 & 0 \end{array} \right) \cdot \left( \begin{array}{c} \xi_{z_{0}} \\ \eta_{z_{0}} \end{array} \right) = \left( \begin{array}{c} -2x_{0}\xi_{z_{0}}-\eta_{z_{0}} \\ \xi_{z_{0}} \end{array} \right) \in S^{u}_{{ \cal H}^{n+1}}$$
\begin{equation} \text{if and only if} \quad \quad |\xi_{z_{0}}| \leq \mu_{h} \cdot |2x_{0}\xi_{z_{0}}+\eta_{z_{0}}| \end{equation}
$$(\text{remember that } (\xi_{z_{0}},\eta_{z_{0}}) \in S^{u}_{{ \cal V}^{n}}, \quad |\eta_{z_{0}}| \leq \mu_{h} |\xi_{z_{0}}| )$$
\\
$$\text{We see that} \quad \quad \mu_{h} \cdot |2x_{0}\xi_{z_{0}}+\eta_{z_{0}}| \geq \mu_{h} \cdot \left[ 2|x_{0}||\xi_{z_{0}}|-|\eta_{z_{0}}| \right] \geq$$
$$\geq \mu_{h} \cdot \left[ 2|x_{0}||\xi_{z_{0}}|-\mu_{h}|\xi_{z_{0}}| \right] = \left[ 2|x_{0}|\mu_{h}-\mu_{h}^{2} \right] |\xi_{z_{0}}| \geq$$
\begin{equation} \geq \left[ 2\mu_{h} \cdot \frac{1}{2} \left( \mu_{h}+\frac{1}{\mu_{h}} \right) -\mu_{h}^{2} \right] |\xi_{z_{0}}| = \left[ \mu_{h}^{2}+1 -\mu_{h}^{2} \right] |\xi_{z_{0}}|=|\xi_{z_{0}}| \end{equation}
\\
and then the inclusion $Df_{n}(S^{u}_{{ \cal V}^{n}})\subset S^{u}_{{ \cal H}^{n+1}}$ is proved.
\newline
\\
Finally for the last part of Assumption 3 we will only prove the inequality
\begin{equation} |\eta_{f_{n}^{-1}(z_{0})}| \geq \frac{1}{\mu} |\eta_{z_{0}}| \quad \text{for } 0<\mu < 1-\mu_{h}\mu_{v} \text{ and } z_{0} \in { \cal H}^{n+1}, \text{ } (\xi_{z_{0}},\eta_{z_{0}}) \in S^{s}_{{ \cal H}^{n+1}} \end{equation}
since the inequality
\begin{equation} |\xi_{f_{n}(z_{0})}| \geq \frac{1}{\mu} |\xi_{z_{0}}| , \quad z_{0} \in { \cal V}^{n}, \text{ } (\xi_{z_{0}},\eta_{z_{0}}) \in S^{u}_{{ \cal V}^{n}} \end{equation}
is proved by using the same argument.
\\
$$|\eta_{f_{n}^{-1}(z_{0})}| = |2y_{0}\eta_{z_{0}}+\xi_{z_{0}}| \geq 2|y_{0}||\eta_{z_{0}}|-|\xi_{z_{0}}| \geq 2|y_{0}||\eta_{z_{0}}|-\mu_{v}|\eta_{z_{0}}|=$$
$$\left[ 2|y_{0}|-\mu_{v} \right] |\eta_{z_{0}}| \geq \frac{1}{\mu} |\eta_{z_{0}}| \quad \quad \text{if and only if}$$
\begin{equation} 2|y_{0}|-\mu_{v} \geq \frac{1}{\mu} \quad \longleftrightarrow \quad |y_{0}| \geq \frac{1}{2} \left( \mu_{v}+\frac{1}{\mu} \right) \end{equation}
\\
This last inequality is true if we require  that $\mu_{v}< \mu <1-\mu_{h}\mu_{v}$. This interval exists since $\mu_{v}=0.615$ is less than $(1-\mu_{h}\mu_{v})=0.621775$. Then we have that
\begin{equation} |y_{0}|>\frac{1}{2} \left( \mu_{v}+\frac{1}{\mu_{v}} \right) > \frac{1}{2} \left( \mu_{v}+\frac{1}{\mu} \right) \quad \text{for every } z_{0}=(x_{0},y_{0}) \in { \cal H}^{n+1}\end{equation}
\\
Analogously for any $z_{0}=(x_{0},y_{0}) \in { \cal V}^{n}$,
\begin{equation} |x_{0}|>\frac{1}{2} \left( \mu_{h}+\frac{1}{\mu_{h}} \right) > \frac{1}{2} \left( \mu_{h}+\frac{1}{\mu} \right) \end{equation}
\\
and the proof that the nonautonomous H\'{e}non map satisfies A1 and A3 with $0< \mu <1- \mu_{h} \mu_{v}$ is complete. Consequently it also satisfies A2 by using Theorem \ref{thm4}.
\newline
\\
By applying the main theorem it follows that there exists a chaotic  invariant  set $\lbrace \Lambda_{n} \rbrace_{n=-\infty}^{+\infty}$ (with respect to the nonautonomous H\'{e}non map $\lbrace f_{n} \rbrace$) contained in $\lbrace D_{n} \rbrace_{n=-\infty}^{+\infty}$ (let say $\Lambda_{n} \subset D_{n}=D$ and $f_{n}(\Lambda_{n})=\Lambda_{n+1}$) which is conjugate to a shift map of two symbols.
\end{proof}

\begin{remark} Comparing this result to what happens for the autonomous H\'{e}non map, it is curious that for some $n \in \mathbb{Z}$ the quantity $A(n)=9.5+\epsilon \cos (n)$ is less than $A_{2}=5+2\sqrt{5} \approx 9.47$, which is the minimum threshold for parameter $A$ for which the autonomous H\'{e}non map satisfies the autonomous Assumption 3.
\newline
\\
In other words, this given example shows that although for some $n \in \mathbb{Z}$ the values that parameter $A$ takes imply that $f_{n}$ does not satisfy the autonomous Assumption 3 separately, this fact does not necessarily mean that the nonautonomous Assumption 3 is not satisfied for the sequence $\lbrace f_{n} \rbrace_{n=-\infty}^{+\infty}$.
\end{remark}

\section{Summary and Conclusions}
\label{sec:summ}

In this paper we have considered a nonautonomous version of the H\'enon map and have provided necessary conditions for the map to possess a nonautonomous chaotic invariant set. This is accomplished by using a nonautonomous  version of the Conley-Moser conditions given in \cite{Wiggins99}.  We sharpen these conditions by providing a more analytical condition that, as a consequence, enables us to show that the chaotic invariant set is hyperbolic. In the course of the proof we provide a precise characterization of what is mean by the phrase ``hyperbolic chaotic invariant set'' for nonautomous dynamical systems. Currently there is much interest in nonautomous dynamics and a thorough analysis of a specific example might provide a benchmark for further studies, just as the work in \cite{Dev79} provided a benchmark for studies of chaotic dynamics for autonomous maps. Indeed, our  generalization of the H\'enon map to the nonautonomous setting provides an approach to generalizing the map to even more general nonautonomous settings, such as a consideration of ``noise''. This would be an interesting topic for future studies.

\vspace{0.3cm}

\section*{\bf Acknowledgments} The research of FB-I, CL and AMM is supported by the MINECO under grant MTM2014-56392-R. The research
of SW is supported by  ONR Grant No.~N00014-01-1-0769.  We acknowledge support from MINECO: ICMAT Severo Ochoa project
SEV-2011-0087.
\vspace{0.3cm}
\\

\end{document}